\documentclass{birkjour}
 \newtheorem{thm}{Theorem}[section]
 
 \newtheorem{lem}[thm]{Lemma}

 \theoremstyle{definition}
 \newtheorem{defn}[thm]{Definition}
 \theoremstyle{remark}
 \newtheorem{rem}[thm]{Remark}

 \usepackage{amsmath}
 \usepackage{amssymb}
 \usepackage[mathscr]{eucal}
 \usepackage[OT1,T1,TS1,T2A]{fontenc}
\usepackage[cp1251]{inputenc}
\usepackage[russian,english]{babel}
 \numberwithin{equation}{section}
\newcommand{\ccomma}{\mathpunct{\raisebox{0.5ex}{,}}}

\DeclareMathOperator{\sech}{sech}

\begin{document}

%-------------------------------------------------------------------------
% editorial commands: to be inserted by the editorial office
%
%\firstpage{1} \volume{228} \Copyrightyear{2004} \DOI{003-0001}
%
%
%\seriesextra{Just an add-on}

\title[The Integral Transform of N.I.\,Akhiezer]
 {The Integral Transform of N.I.\,Akhiezer}
\author[]{Victor Katsnelson}
\address{%
Department of Mathematics\\
The Weizmann Institute\\
76100, Rehovot\\
Israel}
\email{victor.katsnelson@weizmann.ac.il; victorkatsnelson@gmail.com}
%\thanks{}
\date{July 1, 2016}
%%% ----------------------------------------------------------------------
\begin{abstract}
We study the integral transform which appeared in a different form in Akhiezer's textbook
 ``Lectures on Integral Transforms''.
\end{abstract}
\maketitle
\section{The Akhiezer Integral Transforms: a formal definition.}
\label{sec1}
In the present paper we consider the one-parametric family
of pairs \(\boldsymbol{\Phi}_{\omega}, \boldsymbol{\Psi}_{\omega} \) of linear integral operators.
The parameter \(\omega\) which enumerates the family can be an arbitrary positive number and is fixed in the course of our consideration.
Formally the operators \(\boldsymbol{\Phi}_{\omega}, \boldsymbol{\Psi}_{\omega} \) are
defined as convolution operators according the
formulas
\begin{subequations}
\label{FD}
\begin{gather}
\label{FD1}
(\boldsymbol{\Phi}_{\omega}\boldsymbol{x})(t)=
\int\limits_{\mathbb{R}}\varPhi_{\omega}(t-\tau)\boldsymbol{x}(\tau)d\tau,
\quad t\in\mathbb{R},\\
\label{FD2}
(\boldsymbol{\Psi}_{\omega}\boldsymbol{x})(t)=
\int\limits_{\mathbb{R}}\varPsi_{\omega}(t-\tau)\boldsymbol{x}(\tau)d\tau,
\quad t\in\mathbb{R}.
\end{gather}
\end{subequations}
In \eqref{FD}, \(\boldsymbol{x}(\tau)\) is \(2\times1\) vector-column,
\begin{equation}
\label{vf}
\boldsymbol{x}(\tau)=
\begin{bmatrix}
x_1(\tau)\\ x_2(\tau)
\end{bmatrix}
\ccomma
\end{equation}
which entries \(x_1(\tau), x_2(\tau)\) are measurable functions,
and \(\varPhi_{\omega}\), \(\varPsi_{\omega}\) are \(2\times2\) matrices,
\begin{subequations}
\label{FK}
\begin{gather}
\label{FK1}
\varPhi_\omega(t)=
\begin{bmatrix}
\phantom{+}C_{\omega}(t)&
\phantom{+}S_{\omega}(t)\\[1.5ex]
\phantom{+}S_{\omega}(t)&
\phantom{+}C_{\omega}(t)
\end{bmatrix}\ccomma\\[2.5ex]
\label{FK2}
\varPsi_\omega(t)=
\begin{bmatrix}
\phantom{+}C_{\omega}(t)&
-S_{\omega}(t)\\[1.5ex]
-S_{\omega}(t)&
\phantom{+}C_{\omega}(t)
\end{bmatrix}\ccomma
\end{gather}
{\ }
\end{subequations}
where
\begin{equation}
\label{AHF}
C_{\omega}(t)=\frac{\omega}{\pi}\cdot\frac{1}{\cosh\,\omega{}t}\ccomma\quad
S_{\omega}(t)=\frac{\omega}{\pi}\cdot\frac{1}{\sinh\,\omega{}t}\cdot
\end{equation}

\noindent
\vspace{2.0ex}
Here and in what follows, \(\sinh,\,\cosh,\tanh,\sech \) are hyperbolic functions.\\
For \(z\in\mathbb{C}\)
\begin{equation*}
\sinh{}z=\frac{e^z-e^{-z}}{2}\ccomma \
\cosh{}z=\frac{e^z+e^{-z}}{2}\ccomma \  \tanh{}z=\frac{e^z-e^{-z}}{e^z+e^{-z}}\ccomma
\sech{}z=\frac{2}{e^z+e^{-z}}\cdot
\end{equation*}
The operators \(\boldsymbol{\Phi}_{\omega},\,\boldsymbol{\Psi}_{\omega}\) are naturally decomposed
into blocks:
\begin{equation}
\label{bof}
\boldsymbol{\Phi}_{\omega}=
\begin{bmatrix}
\boldsymbol{C}_{\omega}&\boldsymbol{S}_{\omega}\\
\boldsymbol{S}_{\omega}&\boldsymbol{C}_{\omega}
\end{bmatrix}
\ccomma\quad
\boldsymbol{\Psi}_{\omega}=
\begin{bmatrix}
\phantom{-}\boldsymbol{C}_{\omega}&-\boldsymbol{S}_{\omega}\\
-\boldsymbol{S}_{\omega}&\phantom{-}\boldsymbol{C}_{\omega}
\end{bmatrix}\ccomma
\end{equation}
where \(\boldsymbol{C}_{\omega}\) and \(\boldsymbol{S}_{\omega}\) are convolution operators:
\begin{subequations}
\label{coop}
\begin{gather}
\label{coopC}
(\boldsymbol{C}_{\omega}x)(t)=\int\limits_{\mathbb{R}}C_{\omega}(t-\tau)x(\tau)d\tau,\\
\label{coopS}
(\boldsymbol{S}_{\omega}x)(t)=\int\limits_{\mathbb{R}}S_{\omega}(t-\tau)x(\tau)d\tau.
\end{gather}
\end{subequations}
In \eqref{coop}, \(x\) is a \(\mathbb{C}\)-valued function.

The function \(C_{\omega}(\xi)\) is continuous and positive on \(\mathbb{R}\). It decays exponentially
as \(|\xi|\to\infty\):
\begin{equation}
\label{Ec}
\frac{\omega}{\pi}e^{-\omega|\xi|}\leq
C(\xi)<\frac{2\omega}{\pi}e^{-\omega|\xi|},\quad \forall\,\xi\in\mathbb{R}.
\end{equation}
Since
\begin{equation*}
|\tau|-|t|\leq|t-\tau|\leq|\tau|+|t|,
\end{equation*}
the convolution kernel \(C(t-\tau)\) admits the estimate
\begin{equation}
\label{Ecc}
\frac{\omega}{\pi}e^{-\omega|t|}e^{-\omega|\tau|}\leq
C(t-\tau)<\frac{2\omega}{\pi}e^{\omega|t|}e^{-\omega|\tau|},\quad \forall\,t\in\mathbb{R}, \
\forall\,\tau\in\mathbb{R}.
\end{equation}
\begin{defn}
\label{IEW}
 \emph{The set} \(L_{\omega}^{1}\) as the set of all complex valued functions \(x(t)\) which
are measurable, defined almost everywhere with respect to the Lebesgue measure on \(\mathbb{R}\)
and satisfy the condition
\begin{equation}
\label{iew}
\int\limits_{\mathbb{R}}|x(\xi)|e^{-\omega|\xi|}d\xi<\infty.
\end{equation}
\emph{The set} \(L_{\omega}^{1}\dotplus{}L_{\omega}^{1}\) is the set of all \(2\times1\) columns
\(\boldsymbol{x}(t)=\begin{bmatrix}x_1(t)\\ x_2(t)\end{bmatrix}\) such that \(x_1(t)\in{}L_{\omega}^{1}\) and \(x_2(t)\in{}L_{\omega}^{1}\).
\end{defn}
\begin{lem}
\label{LCc}
Let \(x(\tau)\) be a \(\mathbb{C}\)-valued function which belongs to the space \(L_{\omega}^{1}\).
Then the integral in the right hand side of \eqref{coopC} exists\footnote%
{That is
the value of this integral is a \emph{finite} complex number for every \(t\in\mathbb{R}\).}
 for every \(t\in\mathbb{R}\).

We \textsf{define} the function \((\boldsymbol{C}_{\omega}x)(t)\) by means of the equality \eqref{coopC}.
\end{lem}

\begin{rem}
\label{warn1}
For \(x(\tau)\in{}L^1_{\omega}\), the function \((\boldsymbol{C}_{\omega}x)(t)\) is a continuous
function well defined on the whole \(\mathbb{R}\). Nevertheless the function \((\boldsymbol{C}_{\omega}x)(t)\) may not belong to the space \(L^1_{\omega}\). The operator
\(\boldsymbol{C}_{\omega}\) does not map the space \(L^1_{\omega}\) into itself. (In other words, the operator \(\boldsymbol{C}_{\omega}\) considered as an operator in \(L^1_{\omega}\) is
unbounded.)
\end{rem}
The situation with the integral in the right hand side of \eqref{coopS} is more complicated.
The function \(S_{\omega}(\xi)\) also decays exponentially as \(|\xi|\to\infty\):
\begin{equation}
\label{Es}
|S(\xi)|<\frac{2\omega}{\pi(1-e^{-2\omega|\xi|})}e^{-\omega|\xi|},\quad \forall\,\xi\in\mathbb{R}.
\end{equation}
However the function \(S_{\omega}\) has the singularity at the point \(\xi=0\):
\begin{equation}
\label{sg}
S_{\omega}(\xi)=\frac{1}{\pi\xi}+r(\xi),
\end{equation}
where \(r(\xi)\) is a function which is continuous and bounded for \(\xi\in\mathbb{R}\). Thus the
convolution kernel \(S_{\omega}(t-\tau)\) has a non-integrable singularity on the diagonal \(t=\tau\):
\begin{equation*}
\int\limits_{(t-\varepsilon,t+\varepsilon)}|S_{\omega}(t-\tau)|d\tau=\infty,\quad \forall\,t\in\mathbb{R},\,\forall\,\varepsilon>0.
\end{equation*}
Therefore  the integral in the right hand side of \eqref{coopS} may not exist as a Lebesgue integral.
Given a function  \(x(\tau)\), the equality
\begin{equation}
\label{div}
\int\limits_{\mathbb{R}}|S_{\omega}(t-\tau)x(\tau)|d\tau=\infty
\end{equation}
holds at every point \(t\in\mathbb{R}\)  which is a Lebesgue point of the function \(x\) and \(x(t)\not=0\).  Nevertheless, under the condition \eqref{iew} we can attach a meaning to the integral
\(\int\limits_{\mathbb{R}}S_{\omega}(t-\tau)x(\tau)d\tau\) for almost every \(t\in\mathbb{R}\).
\begin{lem}
\label{LSc}
Let \(x(\tau)\) be a \(\mathbb{C}\)-valued function which belongs to the space \(L_{\omega}^{1}\). Then  \emph{the principal value integral}
\begin{equation}
\label{pvi}
 \mathrm{p.v.}\int\limits_{\mathbb{R}}S_{\omega}(t-\tau)x(\tau)d\tau
 \stackrel{\textup{\tiny def}}{=}
\lim\limits_{\varepsilon\to+0}\!\!
\int\limits_{\mathbb{R}\setminus(t-\varepsilon,t+\varepsilon)}\!\!
S_{\omega}(t-\tau)x(\tau)d\tau
\end{equation}
exists for almost every \(t\in\mathbb{R}\).

We \textsf{define} the function \((\boldsymbol{S}_{\omega}x)(t)\) by means of the equality \eqref{coopS},
where  the integral in the right hand side of \eqref{coopS} is interpreted as a principal value integral.
\end{lem}
Under the condition \eqref{iew}, the integral \(\int\limits_{\mathbb{R}\setminus(t-\varepsilon,t+\varepsilon)}\!\!
S_{\omega}(t-\tau)x(\tau)d\tau\) exists as a Lebesgue integral for every \(\varepsilon>0\):
\begin{equation*}
\int\limits_{\mathbb{R}\setminus(t-\varepsilon,t+\varepsilon)}\!\!
|S_{\omega}(t-\tau)x(\tau)|d\tau<\infty,\quad \forall\,t\in\mathbb{R},\,\forall\,\varepsilon>0.
\end{equation*}
This follows from the estimate \eqref{Es}.
The assertion that the limit in \eqref{pvi} exists for almost every \(t\in\mathbb{R}\)
will be proved in section \ref{Hil} using the Hilbert transform theory.
\begin{rem}
\label{warn}
Under the assumption  \eqref{iew}, the function \[y(t)= \textup{p.v.}\int\limits_{\mathbb{R}}S_{\omega}(t-\tau)x(\tau)d\tau,\]
which is defined for almost every \(t\),
 is not necessary locally summable. It may happen that \(\int\limits_{[a,b]}|y(t)|dt=\infty\) for every finite interval \([a,b]\), \(-\infty<a<b<\infty\).
\end{rem}

 Let us define the transforms \(\boldsymbol{\Phi}_{\omega}\) and \(\boldsymbol{\Psi}_{\omega}\)
formally.
\begin{defn}
\label{FoDe}
 For \(\boldsymbol{x}(\tau)=
 \begin{bmatrix}
 x_1(\tau)\\ x_2(\tau)
 \end{bmatrix}
 \in{}L_{\omega}^{1}\dotplus{}L_{\omega}^{1}\), we put
 \begin{equation}
 \label{Fode}
 (\boldsymbol{\Phi}_{\omega}\boldsymbol{x})(t)=\boldsymbol{y}(t),
 \quad
 (\boldsymbol{\Psi}_{\omega}\boldsymbol{x})(t)=\boldsymbol{z}(t),
 \end{equation}
 where \(\boldsymbol{y}(t)\) and \(\boldsymbol{z}(t)\) are \(2\times1\) columns:
 \begin{equation}
 \label{Fode1}
 \boldsymbol{y}(t)=\begin{bmatrix}
 y_1(t)\\ y_2(t)
 \end{bmatrix}
 \ccomma\quad
  \boldsymbol{z}(t)=\begin{bmatrix}
 z_1(t)\\ z_2(t)
 \end{bmatrix}
 \ccomma
  \end{equation}
  with the entries
  \begin{subequations}
  \label{Ent}
  \begin{align}
  \label{Ent1}
 y_1(t)=(\boldsymbol{C}_{\omega}x_1)(t)+(\boldsymbol{S}_{\omega}x_2)(t),&\quad
  z_1(t)=\phantom{-}(\boldsymbol{C}_{\omega}x_1)(t)-(\boldsymbol{S}_{\omega}x_2)(t),\\
  \label{Ent2}
  y_2(t)=(\boldsymbol{S}_{\omega}x_1)(t)+(\boldsymbol{C}_{\omega}x_2)(t),&\quad
  z_2(t)=-(\boldsymbol{S}_{\omega}x_1)(t)+(\boldsymbol{C}_{\omega}x_2)(t).
 \end{align}
 \end{subequations}
 The  operators \(\boldsymbol{C}_{\omega}\) and  \(\boldsymbol{S}_{\omega}\) are the same
 that appeared in Lemmas \ref{LCc} and \ref{LSc} respectively.

According to Lemmas \ref{LCc} and \ref{LSc}, the values
\(\boldsymbol{y}(t)\) and \(\boldsymbol{z}(t)\) are well defined for almost every \(t\in\mathbb{R}\).

 The integral transforms \eqref{FD}-\eqref{FK} are said to be \emph{the Akhiezer integral transforms.}
 \end{defn}
 In another form, these transforms appear in \cite[Chapter 15]{A}.
 (See Problems 3 and 4 to Chapter 15.) The matrix nature of the Akhiezer transforms
 was camouflaged there.

  In what follows, we consider the Akhiezer transform in various functional spaces.
 We show that the operators \(\boldsymbol{\Phi}_{\omega}\) and \(\boldsymbol{\Psi}_{\omega}\)
 are mutually inverse in spaces of functions growing \emph{slower} than \(e^{\omega|t|}\).
 \section{The operators \(\boldsymbol{C}_{\omega}\) and \(\boldsymbol{S}_{\omega}\)  in  \(\boldsymbol{L^2}\).}
 \label{CSL2}
 The Fourier transform machinery is an adequate tool for study convolution operators.\\
\textbf{1}.  Studing the operators \(\boldsymbol{C}_{\omega}\) and \(\boldsymbol{S}_{\omega}\) by means of the Fourier transform technique, we deal with the spaces \(L^1\) and \(L^2\). Both these spaces consist of measurable
 functions defined almost everywhere on the real axis \(\mathbb{R}\) with respect to the Lebesgue measure. The spaces are equipped by the standard linear operations and the standard norms. If \(u\in{}L^1\), then
 \begin{equation}
 \|u\|_{L^1}=\int\limits_{\mathbb{R}}|u(t)|dt.
  \end{equation}
 The space \(L^1\) consists of all \(u\) such that \(\|u\|_{L^1}<\infty\). If \(u\in{}L^2\), then
 \begin{equation}
 \|u\|_{L^2}=\Big\{\int\limits_{\mathbb{R}}|u(\xi)|^{2}d\xi\Big\}^{1/2}.
  \end{equation}
The space \(L^2\) consists of all \(u\) such that \(\|u\|_{L^2}<\infty\).
This space is equipped by inner product \(\langle\,.\,,\,.\,\rangle_{L^2}\).
If \(u^{\prime}\in{}L^2\), \(u^{\prime\prime}\in{}L^2\), then
\begin{equation}
\label{Inpr}
\langle{}u^{\prime},u^{\prime\prime}\rangle_{L^2}=
\int\limits_{\mathbb{R}}u^{\prime}(t)\overline{u^{\prime\prime}(t)}dt.
\end{equation}
\textbf{2}. The Fourier-Plancherel operator \(\mathfrak{F}\):
\begin{equation}
\label{FPd}
\mathfrak{F}u=\hat{u}
\end{equation}
where
\begin{equation}
\label{FTd}
\hat{u}(\lambda)=\int\limits_{\mathbb{R}}u(t)e^{it\lambda}dt,
\end{equation}
maps the space \(L^2\) \emph{on}to itself isometrically:
\begin{equation}
\label{Is}
\|\hat{u}\|_{L^2}^2=2\pi\,\|u\|_{L^2}^2,\quad\forall\,u\in{}L^2.
\end{equation}
The inverse operator \(\mathfrak{F}^{-1}\) is of the form
\begin{equation}
\label{FPi}
\mathfrak{F}^{-1}v=\check{v},
\end{equation}
where
\begin{equation}
\label{FTi}
\check{v}(t)=\frac{1}{2\pi}\int\limits_{\mathbb{R}}v(\lambda)e^{-{}it\lambda}d\lambda.
\end{equation}
\begin{lem}
\label{CoT}
Let \(f\in{}L^2\) and \(k\in{}L^1\). Then
\begin{enumerate}
\item
The integral
\begin{equation}
\label{Cof}
g(t)=\int\limits_{\mathbb{R}}k(t-\tau)f(\tau)d\tau
\end{equation}
 exists as a Lebesgue integral
 (i.e. \(\int\limits_{\mathbb{R}}|k(t-\tau)g(\tau)|d\tau<\infty\)) for almost every \(t\in\mathbb{R}\).
 \item
 The function \(g\) belongs to \(L^2\), and the inequality
 \begin{equation}
 \label{NoE}
 \|g\|_{L^2}\leq\|k\|_{L^1}\|f\|_{L^2}
 \end{equation}
 holds.
 \item
 The Fourier-Plancherel transforms \(\hat{f}\) and \(\hat{g}\) are related by the equality
 \begin{equation}
 \label{fpt}
 \hat{g}(\lambda)=\hat{k}(\lambda)\cdot{}\hat{f}(\lambda), \quad \textup{for a.e.\ } \lambda\in\mathbb{R},
 \end{equation}
 where
 \begin{equation}
 \label{ftk}
 \hat{k}(\lambda)=\int\limits_{\mathbb{R}}k(t)e^{it\lambda}dt,\quad \forall\,\lambda\in\mathbb{R}.
 \end{equation}
 \end{enumerate}
\end{lem}
This lemma can be found in \cite{T}, Theorem 65 there. See also \cite{Bo}, Theorem 3.9.4.

\noindent\textbf{3.}
Let us calculate the Fourier transforms of the functions \(C_{\omega}\) and \(S_{\omega}\).
 The function \(C_{\omega}\) belongs to \(L^1\).
So its Fourier transform
\begin{equation}
\label{FTC}
\widehat{C_{\omega}}(\lambda)=\int\limits_{\mathbb{R}}C_{\omega}(t)e^{{}it\lambda}dt
\end{equation}
 is well defined
for every \(\lambda\in\mathbb{R}\).
\begin{lem}
\label{FTCo}
The Fourier transforms \(\widehat{C_{\omega}}(\lambda)\) of the function \(C_{\omega}(t)\) is:
\begin{equation}
\label{FTc}
\widehat{C_{\omega}}(\lambda)=\sech\dfrac{\pi\lambda}{2\omega}\ccomma
\quad \forall\,\lambda\in\mathbb{R}.
\end{equation}
\end{lem}
 The formula \eqref{FTc} can be found in \cite{T}, where it appears as (7.1.6).

\vspace{3.0ex}
\noindent
\textbf{4}.
The function \(S_{\omega}(t)\) does not belong to \(L^{1}\). This function has
non-integrable singularity at the point \(t=0\). Therefore the integral
\begin{math}
\int\limits_{\mathbb{R}}S_{\omega}(t)e^{2\pi{}it\lambda}dt
\end{math}
does not exist as a Lebesgue integral. However
\begin{math}
\int\limits_{\mathbb{R\setminus(-\varepsilon,\varepsilon)}}|S_{\omega}(t)|dt<\infty,\,\forall\,\varepsilon>0.
\end{math}
 So the integral
 \begin{equation}
 \label{Seps}
 \widehat{S_{\omega,\varepsilon}}(\lambda)=\int\limits_{\mathbb{R}\setminus(-\varepsilon,\varepsilon)}
 S_{\omega}(t)e^{{}it\lambda}dt
\end{equation}
exists as a Lebesgue integral for every \(\varepsilon>0\).
We \emph{define} the Fourier transform  \(\widehat{S_{\omega}}(\lambda)\) as a principle value integral:
\begin{equation}
\label{FTS}
\widehat{S_{\omega}}(\lambda)=\lim\limits_{\varepsilon\to+0}
\int\limits_{\mathbb{R}\setminus(-\epsilon,\epsilon)}S_{\omega}(t)e^{{}it\lambda}dt.
\end{equation}
\begin{lem}
\label{FTSin}
The limit in \eqref{FTS} exists for every \(\lambda\in\mathbb{R}\).
The Fourier transforms \(\widehat{S_{\omega}}(\lambda)\) of the function \(S_{\omega}(t)\) is:
\begin{equation}
\label{FTs}
\widehat{S_{\omega}}(\lambda)=i\cdot\tanh\dfrac{\pi\lambda}{2\omega}\ccomma
\quad \forall\,\lambda\in\mathbb{R}.
\end{equation}
The difference
\begin{equation}
\label{Rem}
\textup{\LARGE\(\varrho\)}_{\omega}(\lambda,\varepsilon)=\widehat{S_{\omega}}(\lambda)-
\widehat{S_{\omega,\varepsilon}}(\lambda), \quad\forall\,\lambda\in\mathbb{R},\,\varepsilon>0.
\end{equation}
satisfies the conditions
\begin{gather}
\label{LimR}
\lim\limits_{\varepsilon\to+0}\textup{\LARGE\(\varrho\)}_{\omega}(\lambda,\varepsilon)=0,\quad
\forall\,\lambda\in\mathbb{R},\\
\intertext{and}
\label{RemB}
\sup_{\substack{\lambda\in\mathbb{R},\\
0<\varepsilon\leq\frac{\pi}{4\omega}}}\big|\textup{\LARGE\(\varrho\)}_{\omega}(\lambda,\varepsilon)\big|<\infty.
\end{gather}
\end{lem}
 The formula \eqref{FTs} can be found in \cite{T}, where it appears as (7.2.3).\\

\noindent
\textbf{5}. In Section \ref{sec1} we already have defined the functions \(\boldsymbol{C}_{\omega}x\)
and \(\boldsymbol{S}_{\omega}x\) for \(x\) from the space \(L^1_{\omega}\). The space \(L^2\) is contained in \(L^1_{\omega}\). If \(x\in{}L^2\), then
\begin{equation}
\label{sub}
\int\limits_{\mathbb{R}}|x(t)|e^{-\omega|t|}dt\leq
\Big\{\int\limits_{\mathbb{R}}|x(t)|^2dt\Big\}^{1/2}
\Big\{\int\limits_{\mathbb{R}}e^{-2\omega|t|}dt\Big\}^{1/2}<\infty.
\end{equation}
According to Lemmas \ref{LCc} and \ref{LSc}, if \(f\in{}L^1_{\omega}\), then the function \((\boldsymbol{C}_{\omega}f)(t)\)
is defined for every \(t\in\mathbb{R}\) and the function \((\boldsymbol{S}_{\omega}f)(t)\)
is defined for almost every \(t\in\mathbb{R}\). However for \(f\in{}L^2(\mathbb{R})\), we can obtain much more accurate
results.
\begin{lem}
\label{FGc}
Let \(f\in{}L^2\) and \(g=\boldsymbol{C}_{\omega}f\), i.e.
\begin{equation}
\label{fgc}
g(t)=\int\limits_{\mathbb{R}}C_{\omega}(t-\tau)f(\tau)d\tau.
\end{equation}
Then \(g\in{}L^2\), and the Fourier-Plancherel transforms \(\hat{f}\),
\(\hat{g}\) of functions \(f\) and \(g\) are related by the equality
\begin{equation}
\label{fcr}
\hat{g}(\lambda)=\widehat{C_{\omega}}(\lambda)\cdot\hat{f}(\lambda),\quad
\mathrm{a.e.\ on\ } \mathbb{R},
\end{equation}
where \(\widehat{C_{\omega}}(\lambda)\) is determined by the equality \eqref{FTc}.
\end{lem}
\begin{proof} Lemma \ref{FGc} is a direct consequence  of Lemma \ref{CoT}.

\end{proof}
\begin{lem}
\label{FGs}
Let \(f\in{}L^2\) and \(g=\boldsymbol{S}_{\omega}f\), i.e.
\begin{equation}
\label{fgs}
g(t)=\mathrm{p.v.}\int\limits_{\mathbb{R}}S_{\omega}(t-\tau)f(\tau)d\tau.
\end{equation}
Then \(g\in{}L^2\), and the Fourier-Plancherel transforms \(\hat{f}\),
\(\hat{g}\) of functions \(f\) and \(g\) are related by the equality
\begin{equation}
\label{fsr}
\hat{g}(\lambda)=\widehat{S_{\omega}}(\lambda)\cdot\hat{f}(\lambda),\quad
\mathrm{a.e.\ on\ } \mathbb{R},
\end{equation}
where \(\widehat{S_{\omega}}(\lambda)\) is determined by the equality \eqref{FTs}.
\end{lem}
\begin{proof} Since \(S_{\omega}\not\in{}L^1\), Lemma \ref{FGs} does not follow from Lemma \ref{CoT} directly. Let
\begin{equation}
\label{SReg}
S_{\omega,\varepsilon}(t)=
\begin{cases}
S_{\omega}(t),&\textrm{if} \ t\in\mathbb{R}\setminus(-\varepsilon,\varepsilon),\\
0\phantom{_{\omega}(t)},&\textrm{if} \ t\in(-\varepsilon,\varepsilon).
\end{cases}
\end{equation}
The function \(S_{\omega,\varepsilon}\) belongs to \(L^1\) for every \(\varepsilon>0\). Let
\begin{equation}
\label{geps}
g_{\varepsilon}(t)=\int\limits_{\mathbb{R}}S_{\omega,\varepsilon}(t-\tau)f(\tau)d\tau.
\end{equation}
Applying Lemma \ref{CoT}  to \(k=S_{\omega,\varepsilon}\), we conclude that
\(g_{\varepsilon}\in{}L^2\) for every \(\varepsilon>0\) and that the Fourier-Plancherel transforms \(\widehat{g_{\varepsilon}}\), \(\widehat{f}\) of the functions
\(g_{\varepsilon}\), \(f\) are related by the equality
\begin{equation}
\label{reeq}
\widehat{g_{\varepsilon}}(\lambda)=\widehat{S_{\omega,\varepsilon}}(\lambda)\cdot\hat{f}(\lambda),
\end{equation}
where \(\widehat{S_{\omega,\varepsilon}}(\lambda)\) is defined by \eqref{Seps}. According to Lemma \ref{FTSin},
\begin{equation}
\widehat{g_{\varepsilon}}(\lambda)=
\widehat{S_{\omega}}(\lambda)\cdot\hat{f}(\lambda)-h_{\varepsilon}(\lambda),
\end{equation}
where
\begin{equation}
\label{Remb}
h_{\varepsilon}(\lambda)=\textup{\LARGE\(\varrho\)}_{\omega}(\lambda,\varepsilon)\hat{f}(\lambda),
\end{equation}
and the family \(\{\textup{\LARGE\(\varrho\)}_{\omega}(\lambda,\varepsilon)\}_{0<\varepsilon<\infty}\)
satisfies the conditions \eqref{LimR} and \eqref{RemB}.
From \eqref{LimR}, \eqref{RemB}, \eqref{Remb} and the Lebesgue Dominated Convergence Theorem it follows that
\begin{equation*}
\lim\limits_{\varepsilon\to+0}\int\limits_{\mathbb{R}}|h_{\varepsilon}(\lambda)|^2d\lambda=0.
\end{equation*}
In other words,
\begin{equation}
\label{gest}
\|\widehat{g_{\varepsilon}}(\lambda)-
\hat{g}\|_{L^2}=0,
\end{equation}
where
\begin{equation}
\hat{g}\stackrel{\mathrm{def}}{=}\widehat{S_{\omega}}(\lambda)\cdot\hat{f}(\lambda).
\end{equation}
From \eqref{gest} it follows that \(\|g_{\varepsilon}-\check{g}\|_{L^2}\to0\) as
\(\varepsilon\to+0\), i.e.
\begin{equation}
\label{lif2}
\|S_{\omega,\varepsilon}f-\check{g}\|_{L^2}\to0\quad\mathrm{as }\ \varepsilon\to+0,
\end{equation}
where \(\check{g}=\mathfrak{F}^{-1}\hat{g}\in{}L^2\). From the other side, \((S_{\omega,\varepsilon}f)(t)\to{}g(t)\) for a.e. \(t\in\mathbb{R}\) by Lemma \ref{LSc}.
Hence \(g=\check{g}\), and \(\hat{g}=\widehat{S_{\omega}}\hat{f}\).
\end{proof}
\noindent
\textbf{6}. The equality
\begin{equation}
\label{CruEq}
\big|\widehat{C_{\omega}}(\lambda)\big|^2+\big|\widehat{S_{\omega}}(\lambda)\big|^2=1,
\quad \forall\,\lambda\in\mathbb{R},
\end{equation}
plays a crucial role in this paper. This equation is a direct consequence of the explicite expressions \eqref{FTc} and \eqref{FTs} for \(\widehat{C_{\omega}}\) and \(\widehat{S_{\omega}}\)
and the identity
\begin{equation}
\label{TrId}
(\cosh\zeta)^2-(\sinh\zeta)^2=1,\quad \forall\zeta\in\mathbb{C}.
\end{equation}
\begin{lem}
\label{Contr}
The operators \(\boldsymbol{C}_{\omega}\) and \(\boldsymbol{S}_{\omega}\) are contractive
in the space \(L^2(\mathbb{R})\). Moreover the equality
\begin{equation}
\label{Pif}
\|\boldsymbol{C}_{\omega}f\|^2_{L^2}+\|\boldsymbol{S}_{\omega}f\|^2_{L^2}=
\|f\|^2_{L^2},\quad\forall\,f\in{}L^2,
\end{equation}
holds.
\end{lem}
\begin{proof}
Let \(g_c=\boldsymbol{C}_{\omega}f\), \(g_s=\boldsymbol{S}_{\omega}f\) and let
\(\hat{f}\), \(\widehat{g_c}\), \(\widehat{g_s}\) be the Fourier-Plancherel transforms of the
functions \(f,\,g_c,\,g_s\). According to Lemmas \ref{FGc} and \ref{FGs}, the equalities
\begin{equation*}
\widehat{g_c}(\lambda)=\widehat{C_{\omega}}(\lambda)\hat{f}(\lambda),\quad
\widehat{g_s}(\lambda)=\widehat{S_{\omega}}(\lambda)\hat{f}(\lambda),\quad
\mathrm{for\  a.e.\ }\lambda\in\mathbb{R}
\end{equation*}
hold. From \eqref{CruEq} it follows that
\begin{equation*}
|\widehat{g_c}(\lambda)|^2+|\widehat{g_s}(\lambda)|^2=|\hat{f}(\lambda)|^2, \quad
\mathrm{for\  a.e.\ }\lambda\in\mathbb{R}
\end{equation*}
Integrating with respect to \(\lambda\), we obtain
the equality \(\|\widehat{g_c}\|^2_{L^2}+\|\widehat{g_s}\|^2_{L^2}=
\|\hat{f}\|^2_{L^2}\). In view of \eqref{Is}, the last equality is equivalent to the equality
\eqref{Pif}.
\end{proof}
\section{The Akhiezer operators \(\boldsymbol{\Phi}_{\omega}\) and \(\boldsymbol{\Psi}_{\omega}\) in \(\boldsymbol{L^2\oplus{}L^2}\).}
\label{AO}
\begin{defn}
\label{OrtS}
\emph{The space} \(L^{2}\oplus{}L^{2}\) is the set of all \(2\times1\) columns
\(\boldsymbol{x}=\begin{bmatrix}x_1\\ x_2\end{bmatrix}\) such that
\(x_1(t)\in{}L^{2}\) and \(x_2(t)\in{}L^{2}\). The set
\(L^{2}\oplus{}L^{2}\) is equipped by the natural linear operations
and by the inner product \(\langle\,,\,\rangle_{L^{2}\oplus{}L^{2}}\).\\ If
\(\boldsymbol{x^{\prime}}(t)=\begin{bmatrix}x_1^{\prime}(t)\\ x_2^{\prime}(t)\end{bmatrix}\)
and \(\boldsymbol{x^{\prime\prime}}(t)=\begin{bmatrix}x_1^{\prime\prime}(t)\\ x_2^{\prime\prime}(t)\end{bmatrix}\) belong to \(L^{2}\oplus{}L^{2}\),
then
\begin{equation}
\label{InP}
\langle{}\boldsymbol{x^{\prime}},\boldsymbol{x^{\prime\prime}}\rangle
_{L^{2}\oplus{}L^{2}}\stackrel{\mathrm{def}}{=}
\langle{}x_1^{\prime},x_1^{\prime\prime}\rangle_{L^2}+
\langle{}x_2^{\prime},x_2^{\prime\prime}\rangle_{L^2}.
\end{equation}
\end{defn}

\noindent
\vspace{3.0ex}
The inner product \eqref{InP} generates the norm
\begin{equation}
\label{NoP}
\|\boldsymbol{x}\|_{L^{2}\oplus{}L^{2}}=
\sqrt{\|x_1\|^2_{L^2}+\|x_2\|^2_{L^2}}\quad\mathrm{for} \quad
\boldsymbol{x}=\begin{bmatrix}x_1\\ x_2\end{bmatrix} \in{}L^2\oplus{}L^2.
\end{equation}
Since\footnote{See \eqref{sub}.} \(L^2\subset{}L^1_{\omega}\), also \(L^2\oplus{}L^2\subset{}L^1_{\omega}\dotplus{}L^1_{\omega}\). Thus if
\(\boldsymbol{x}\in{}L^2\oplus{}L^2\), then the values
\begin{math}
\boldsymbol{y}(t)=(\boldsymbol{\Phi}_{\omega}\boldsymbol{x})(t)
\end{math}
and
\begin{math}
\boldsymbol{z}(t)=(\boldsymbol{\Psi}_{\omega}\boldsymbol{x})(t)
\end{math}
are defined by \eqref{Ent} for almost every \(t\in\mathbb{R}\). Using Lemmas \ref{FGc} and \ref{FGs}, we conclude from \eqref{Ent}  that the operators \(\boldsymbol{\Phi}_{\omega}\) and \(\boldsymbol{\Psi}_{\omega}\) are bounded operators in
the space \(L^2\oplus{}L^2\). In particular, the values \(\boldsymbol{y}(t)\) and \(\boldsymbol{z}(t)\) belong to \(L^2\oplus{}L^2\).
\begin{thm}
\label{IsO}
Each of the operators \(\boldsymbol{\Phi}_{\omega}\) and \(\boldsymbol{\Psi}_{\omega}\) is an isometric operator in
the space \(L^2\oplus{}L^2\):
\begin{multline}
\label{iso}
\hfil
\|\boldsymbol{\Phi}_{\omega}\boldsymbol{x}\|_{L^2\oplus{}L^2}=
\|\boldsymbol{x}\|_{L^2\oplus{}L^2},\quad \|\boldsymbol{\Psi}_{\omega}\boldsymbol{x}\|_{L^2\oplus{}L^2}=
\|\boldsymbol{x}\|_{L^2\oplus{}L^2},\\ \forall\,\boldsymbol{x}\in{}L^2\oplus{}L^2.
\end{multline}
\end{thm}
\begin{thm}
\label{muin}
The operators \(\boldsymbol{\Phi}_{\omega}\) and \(\boldsymbol{\Psi}_{\omega}\) are mutually inverse in the space \(L^2\oplus{}L^2\):
\begin{subequations}
\label{MUI}
\begin{align}
\label{mui1}
\boldsymbol{\Psi}_{\omega}\boldsymbol{\Phi}_{\omega}\boldsymbol{x}=\boldsymbol{x},&\qquad\forall\,\boldsymbol{x}\in{}L^2\oplus{}L^2,\\
\label{mui2}
\boldsymbol{\Phi}_{\omega}\boldsymbol{\Psi}_{\omega}\boldsymbol{x}=\boldsymbol{x},&\qquad\forall\,\boldsymbol{x}\in{}L^2\oplus{}L^2.
\end{align}
\end{subequations}
\end{thm}
\begin{proof}[Proofs of Theorem \ref{IsO}] Let us associate the \(2\times2\)
matrix functions \(\widehat{{\varPhi}_{\omega}}(\lambda)\) and
\(\widehat{{\varPsi}_{\omega}}(\lambda)\) with the operators \(\boldsymbol{\Phi}_{\omega}\) and \(\boldsymbol{\Psi}_{\omega}\):
\begin{subequations}
\label{FKF}
\begin{gather}
\label{FK1F}
\widehat{\varPhi_\omega}(\lambda)=
\begin{bmatrix}
\phantom{+}\widehat{C_{\omega}}(\lambda)&
\phantom{+}\widehat{S_{\omega}}(\lambda)\\[1.5ex]
\phantom{+}\widehat{S_{\omega}}(\lambda)&
\phantom{+}\widehat{C_{\omega}}(\lambda)
\end{bmatrix}\ccomma\quad\lambda\in\mathbb{R},\\[2.5ex]
\label{FK2F}
\widehat{\varPsi_\omega}(\lambda)=
\begin{bmatrix}
\phantom{+}\widehat{C_{\omega}}(\lambda)&
-\widehat{S_{\omega}}(\lambda)\\[1.5ex]
-\widehat{S_{\omega}}(\lambda)&
\phantom{+}\widehat{C_{\omega}}(\lambda)
\end{bmatrix}\ccomma\quad\lambda\in\mathbb{R},
\end{gather}
\end{subequations}
where \(\widehat{C_{\omega}}(\lambda)\) and \(\widehat{S_{\omega}}(\lambda)\) are the same that in \eqref{FTc} and \eqref{FTs}.
Let
 \[\boldsymbol{x}=\begin{bmatrix}x_1\\ x_2\end{bmatrix}\in{}L^2\oplus{}L^2, \quad \boldsymbol{y}=\begin{bmatrix}y_1\\ y_2\end{bmatrix}=\boldsymbol{\Phi}_{\omega}\boldsymbol{x},\]
and let
\(\widehat{\boldsymbol{x}}=\begin{bmatrix}\widehat{x_1}\\ \widehat{x_2}\end{bmatrix}\),
\(\widehat{\boldsymbol{y}}=\begin{bmatrix}\widehat{y_1}\\ \widehat{y_2}\end{bmatrix}\), where
\(\widehat{x_1}, \widehat{x_2}, \widehat{y_1}, \widehat{y_2}\) are the Fourier-Plancherel transforms of the functions \(x_1, x_2, y_1, y_2\) respectively. According to
the equality \eqref{Ent} and to Lemmas \ref{FGc}
and \ref{FGs}, the equality
\begin{equation}
\label{FIm}
\widehat{\boldsymbol{y}}(\lambda)=\widehat{{\varPhi}_{\omega}}(\lambda)
\widehat{\boldsymbol{x}}(\lambda)
\end{equation}
holds for almost every \(\lambda\in\mathbb{R}\).

 From the equality \eqref{CruEq} it follows that  the matrix
\(\widehat{{\varPhi}_{\omega}}(\lambda)\) is unitary for each \(\lambda\in\mathbb{R}\):
\begin{equation}
\label{Un}
\big(\widehat{{\varPhi}_{\omega}}(\lambda)\big)^{\ast}\widehat{{\varPhi}_{\omega}}(\lambda)=I,\quad\forall\,\lambda\in\mathbb{R},
\end{equation}
where \(I\) is \(2\times2\) identity matrix.
From \eqref{FIm} and \eqref{Un} it follows that
\begin{equation*}
\big(\widehat{\boldsymbol{y}}(\lambda)\big)^{\ast}\widehat{\boldsymbol{y}}(\lambda)=
\big(\widehat{\boldsymbol{x}}(\lambda)\big)^{\ast}\widehat{\boldsymbol{x}}(\lambda),\quad \textup{for a.e.\ }\lambda\in\mathbb{R},
\end{equation*}
i.e.
\begin{equation*}
|\widehat{y}_1(\lambda)|^2+|\widehat{y}_2(\lambda)|^2=|\widehat{x}_1(\lambda)|^2+|\widehat{x}_2(\lambda)|^2,\quad
\textup{for a.e.\ }\lambda\in\mathbb{R},
\end{equation*}
Integrating with respect to \(\lambda\) over \(\mathbb{R}\) and using the Parseval identity \eqref{Is}, we conclude that
\begin{equation*}
\|y_1\|^2_{L^2}+\|y_2\|^2_{L^2}=\|x_1\|^2_{L^2}+\|x_2\|^2_{L^2},
\end{equation*}
that is \(\|\boldsymbol{\Phi}_{\omega}\boldsymbol{x}\|_{L^2\oplus{}L^2}=
\|\boldsymbol{x}\|_{L^2\oplus{}L^2}\). The equality \(\|\boldsymbol{\Psi}_{\omega}\boldsymbol{x}\|_{L^2\oplus{}L^2}=
\|\boldsymbol{x}\|_{L^2\oplus{}L^2}\) can be obtained analogously.
\end{proof}
\begin{proof}[Proof of Theorem \ref{iso}] Let
\begin{equation*}
\boldsymbol{x}=\begin{bmatrix}x_1\\ x_2\end{bmatrix}\in{}L^2\oplus{}L^2, \quad \boldsymbol{y}=\begin{bmatrix}y_1\\ y_2\end{bmatrix}=\boldsymbol{\Phi}_{\omega}\boldsymbol{x},\quad \boldsymbol{z}=\begin{bmatrix}z_1\\ z_2\end{bmatrix}=\boldsymbol{\Psi}_{\omega}\boldsymbol{y}.
\end{equation*}
Let \(\widehat{\boldsymbol{x}}=\begin{bmatrix}\widehat{x_1}\\ \widehat{x_2}\end{bmatrix}\),
\(\widehat{\boldsymbol{y}}=\begin{bmatrix}\widehat{y_1}\\ \widehat{y_2}\end{bmatrix}\),
\(\widehat{\boldsymbol{z}}=\begin{bmatrix}\widehat{z_1}\\ \widehat{z_2}\end{bmatrix}\), where
\(\widehat{x_1}, \widehat{x_2}, \widehat{y_1}, \widehat{y_2}, \widehat{z_1}, \widehat{z_2}\) are the Fourier-Plancherel transforms of the functions \(x_1, x_2, y_1, y_2, z_1, z_2\) respectively. We already proved the equality \eqref{FIm}. In the same way the equality
\begin{equation}
\label{IItw}
\widehat{\boldsymbol{z}}(\lambda)=\widehat{{\varPsi}_{\omega}}(\lambda)
\widehat{\boldsymbol{y}}(\lambda),\quad \textup{for a.e. \ }\lambda\in\mathbb{R},
\end{equation}
can be established. From \eqref{FIm} and \eqref{IItw} it follows that
\begin{equation}
\label{Con}
\widehat{\boldsymbol{z}}(\lambda)=\widehat{{\varPsi}_{\omega}}(\lambda)\widehat{{\varPhi}_{\omega}}(\lambda)
\widehat{\boldsymbol{x}}(\lambda),\quad \textup{for a.e. \ }\lambda\in\mathbb{R},
\end{equation}
 From the equality \eqref{CruEq} it follows that the matrices \(\varPhi_{\omega}(\lambda)\) and \(\varPsi_{\omega}(\lambda)\) are mutually inverse:
\begin{equation}
\label{MuIn}
\varPhi_{\omega}(\lambda)\varPsi_{\omega}(\lambda)=I,\quad\forall\,\lambda\in\mathbb{R},
\end{equation}
where \(I\) is \(2\times2\) identity matrix. From \eqref{Con} and \eqref{MuIn} we conclude that
\begin{equation*}
\widehat{\boldsymbol{z}}(\lambda)=
\widehat{\boldsymbol{x}}(\lambda),\quad \textup{for a.e. \ }\lambda\in\mathbb{R}.
\end{equation*}
Finally \(\boldsymbol{z}=\boldsymbol{x}\).

The equality \eqref{mui1} is proved. The equality \eqref{mui2} can be proved in the same way.
\end{proof}
\section{The Hilbert transform}
\label{Hil}
\begin{defn}
Let \(u(\tau)\) be a complex-valued function which is defined for almost every \(\tau\in\mathbb{R}\). We assume that the function \(u\) satisfies the condition
\begin{equation}
\label{Cc}
\int\limits_{\mathbb{R}}\frac{|u(\tau)|}{1+|\tau|}d\tau<\infty.
\end{equation}
 Then the integral
\begin{equation}
\label{PH}
\displaystyle
H_{\varepsilon}u(t)= \frac{1}{\pi}\!\!\!
\int\limits_{\mathbb{R}\setminus(t-\varepsilon,t+\varepsilon)}
\frac{u(\tau)}{t-\tau}d\tau
\end{equation}
exists for every \(t\in\mathbb{R}\) and \(\varepsilon>0\). For each \(\varepsilon>0\), the function \(H_{\varepsilon}u(t)\) is a continuous function of
\(t\) for \(t\in\mathbb{R}\).
The function \(Hu(t)\) is defined for those \(t\in\mathbb{R}\) for which
the value \(H_{\varepsilon}u(t)\) tends to a finite limit as \(\varepsilon\to+0\):
\begin{equation}
\label{DH}
Hu(t)\stackrel{\textup{\tiny{def}}}{=}\lim\limits_{\varepsilon\to+0}
 \frac{1}{\pi}\!\!\!
\int\limits_{\mathbb{R}\setminus(t-\varepsilon,t+\varepsilon)}
\frac{u(\tau)}{t-\tau}d\tau.
\end{equation}
The function \(Hu\) is said to be the Hilbert transform of the function \(u\).
\end{defn}

\vspace{2.0ex}
\noindent
\textbf{Theorem. (A.I.Plessner.)} \textit{Let \(u(\tau)\) be a function which is defined for almost every
\(\tau\in\mathbb{R}\).
 If the function \(u(\tau)\) satisfies the condition \eqref{Cc}, then its Hilbert
transform \(Hu(t)\)  exists for almost every \(t\in\mathbb{R}\).}

\vspace{2.0ex}
\noindent
Proof of this Plessner's Theorem can be found in \cite{T}, Theorem 100 there.

\vspace{2.0ex}
If \(u\) is a function from \(L^2\),
then \(u\) satisfies the condition \eqref{Cc}. By Plessner's theorem, the Hilbert
transform \(v(t)=(Hu)(t)\) exists for almost every \(t\in\mathbb{R}\).

\vspace{2.0ex}
\noindent
\textbf{Theorem. (E.C.\,Titchmarch.)} \textit{Let \(u\) be a function from \(L^2\).
Then:}
\begin{enumerate}
\item
 \textit{Its Hilbert transform \(v=Hu\) also belongs to \(L^2\), and the equality
\begin{equation}
\label{PaH}
\|v\|_{L^2}=\|u\|_{L^2}
\end{equation}
holds.}
\item
\textit{The equality
\begin{equation}
\label{IO}
(Hv)(t)=-u(t)
\end{equation}
holds for almost every \(t\in\mathbb{R}\).}
\end{enumerate}

This theorem means that the Hilbert transform, considered as an operator  in \(L^2\), is an unitary operator which satisfies the equality
\begin{equation}
\label{UI}
H^2=-{I},
\end{equation}
where \({I}\) is the identity operator in \(L^2\).

\begin{proof}[Proof of Lemma \ref{LSc}.] We use the decomposition \eqref{sg} of the kernel \mbox{\(S(t-\tau)\)} into the sum of the Hilbert kernel
\(\dfrac{1}{\pi(t-\tau)}\) and the `regular' kernel \(r(t-\tau)\).
 Let \((a,b)\subset\mathbb{R}\) be an \emph{arbitrary} finite interval of the real axis.  We split the function \(f(\tau)\) into the sum of two summands.
\begin{gather}
\label{Spl}
f(\tau)=g(\tau)+h(\tau),\\
\intertext{where}
\label{g}
g(\tau)=
\begin{cases}
f(\tau),& \quad \textup{if} \ \tau\in \phantom{\mathbb{R}\ \setminus}(a,b), \\
\phantom{f} 0  \ \    , & \quad \textup{if} \ \tau\in\mathbb{R}\setminus(a,b).
\end{cases}
\end{gather}
So
\begin{gather}
\label{h}
h(\tau)=0,\quad    \textup{if} \ \tau\in(a,b).
\end{gather}
According to \eqref{sg} and \eqref{Spl}, the equality
\begin{equation}
\label{SpE}
\int\limits_{\mathbb{R}\setminus(t-\varepsilon,t+\varepsilon)}\!\!
S_{\omega}(t-\tau)f(\tau)d\tau=I_{1,\varepsilon}(t)+
I_{2,\varepsilon}(t)+I_{3,\varepsilon}(t)
\end{equation}
holds, where
\begin{gather}
\label{I1}
I_{1,\varepsilon}(t)=\int\limits_{\mathbb{R}\setminus(t-\varepsilon,t+\varepsilon)}\!\!
\frac{1}{\pi(t-\tau)}g(\tau)d\tau,\\
\label{I2}
I_{2,\varepsilon}(t)=\int\limits_{\mathbb{R}\setminus(t-\varepsilon,t+\varepsilon)}\!\!
r(t-\tau)g(\tau)d\tau,\\
\label{I3}
I_{3,\varepsilon}(t)=\int\limits_{\mathbb{R}\setminus(t-\varepsilon,t+\varepsilon)}\!\!
S_{\omega}(t-\tau)h(\tau)d\tau.
\end{gather}
The function \(g\) satisfies the condition \eqref{Cc}. According Plessner's Theorem,
\(\lim\limits_{\varepsilon\to+0}I_{1,\varepsilon}(t)\) exists for almost every
\(t\in\mathbb{R}\).  Since the function \(g\) is finitely supported and the kernel
\(r(t-\tau)\) is continuous, \(\lim\limits_{\varepsilon\to+0}I_{2,\varepsilon}(t)\) exists for every \(t\in\mathbb{R}\). Since the function \(h(\tau)\) vanishes for
\(\tau\in(a,b)\) and
\(\int\limits_{\mathbb{R}}\frac{|h(\tau)|}{\cosh{}\omega\tau}d\tau<\infty\),
\(\lim\limits_{\varepsilon\to+0}I_{3,\varepsilon}(t)\) exists for every \(t\in(a,b)\). In view of \eqref{SpE}, the limit in \eqref{pvi} exists
for almost every \(t\in(a,b)\). Since \((a,b)\) is an arbitrary finite interval,
the limit in \eqref{pvi} exists for almost every \(t\in\mathbb{R}\).
\end{proof}
 \section{The operators \(\boldsymbol{C}_{\omega}\) and \(\boldsymbol{S}_{\omega}\)  in  \(\boldsymbol{L^2_\sigma}\).}
 In this section we consider the operators   \(\boldsymbol{C}_{\omega}\) and \(\boldsymbol{S}_{\omega}\)  acting in spaces of functions growing slower than \(e^{\omega|t|}\) as \(t\to\pm\infty\).

\begin{defn}
\label{Lsig}
For \(\sigma\in\mathbb{R}\), the space \(L^2_{\sigma}\) is the space of all functions \(x\)
which are measurable, defined almost everywhere with respect to the Lebesgue measure and satisfy
the condition \(\|x\|_{L^2_{\sigma}}<\infty\), where
\begin{equation}
\label{Nsi}
\|x\|_{L^2_{\sigma}}=\bigg\{\int\limits_{\mathbb{R}}|x(\tau)|^2e^{-2\sigma|\tau|}d\tau\bigg\}^{1/2}\,.
\end{equation}
The space \(L^2_{\sigma}\) is equipped by the standard linear operations and by the norm \eqref{Nsi}.
\end{defn}

It is clear that the space \(L^2\) which appeared in section \ref{CSL2} is the space \(L^2_{_0}\),
that is \(L^2_{\sigma}\) with \(\sigma=0\).

 In Section \ref{sec1} we already have defined the functions \(\boldsymbol{C}_{\omega}x\)
and \(\boldsymbol{S}_{\omega}x\) for \(x\) from the space \(L^1_{\omega}\). For \(\sigma<\omega\), the space \(L^2_{\sigma}\) is contained in \(L^1_{\omega}\). If \(x\in{}L^2_{\sigma}\), then
\begin{equation}
\label{subb}
\int\limits_{\mathbb{R}}|x(t)|e^{-\omega|t|}dt\leq
\Big\{\int\limits_{\mathbb{R}}|x(t)|^2e^{-2\sigma|t|}dt\Big\}^{1/2}
\Big\{\int\limits_{\mathbb{R}}e^{-2(\omega-\sigma)|t|}dt\Big\}^{1/2}<\infty.
\end{equation}
According to Lemmas \ref{LCc} and \ref{LSc}, if \(f\in{}L^1_{\omega}\), then the function \((\boldsymbol{C}_{\omega}f)(t)\)
is defined for every \(t\in\mathbb{R}\) and the function \((\boldsymbol{S}_{\omega}f)(t)\)
is defined for almost every \(t\in\mathbb{R}\).

In Section \ref{CSL2} we obtained that if \(f\in{}L^2\),   than
\(\boldsymbol{C}_{\omega}f\in L^2\) and \(\boldsymbol{S}_{\omega}f\in L^2\). Moreover we proved that the operators \(\boldsymbol{C}_{\omega}\) and
\(\boldsymbol{S}_{\omega}\) are contractive in \(L^2\): see Corollary \ref{Contr}.
In this section we  show that if \(0<\sigma<\omega\) and \(f\in{}L^2_{\sigma}\), than
\(\boldsymbol{C}_{\omega}f\in L^2_{\sigma}\) and \(\boldsymbol{S}_{\omega}f\in L^2_{\sigma}\).
Moreover we show that the operators \(\boldsymbol{C}_{\omega}\) and
\(\boldsymbol{S}_{\omega}\) are bounded in the space \(L^2_{\sigma}\).
\begin{lem}
\label{Boc}
Assume that \(0\leq\sigma<\omega\). Let  \(f\in{}L^2_{\sigma}\), and \(g\) is related to \(f\)
by means of the formula \eqref{fgc}, i.e. \(g=\boldsymbol{C}_{\omega}f\). Then
\(g\in{}L^2_{\sigma}\), and
\begin{equation}
\label{BOc}
\|g\|_{L^2_{\sigma}}\leq\frac{M_c}{1-\sigma/\omega}\,\|f\|_{L^2_{\sigma}},
\end{equation}
where \(M_c<\infty\) is a value which does not depend on  \(\omega\) and \(\sigma\).
\end{lem}
\begin{proof}
Let
\begin{equation}
\label{Sc}
u(\tau)=f(\tau)e^{-\sigma|\tau|},\quad v(t)=g(t)e^{-\sigma|t|}.
\end{equation}
Since \(f\in{}L^2_{\sigma}\), \(u\in{}L^2\).
The equality \eqref{fgc} can be rewritten as
\begin{equation}
\label{Rew}
v(t)=\int\limits_{\mathbb{R}}e^{-\sigma|t|+\sigma|\tau|}C_{\omega}(t-\tau)u(\tau)d\tau.
\end{equation}
Let us estimate the kernel
\begin{equation}
\label{Kc}
K_c(t,\tau)=e^{-\sigma|t|+\sigma|\tau|}\,C_{\omega}(t-\tau).
\end{equation}
For \(\sigma\geq0\) the inequality
\begin{equation}
\label{TrI}
\big|-\sigma|t|+\sigma|\tau|\big|\leq\sigma|t-\tau\big|,\quad\forall\,t\in\mathbb{R},\,\tau\in\mathbb{R},
\end{equation}
holds. Hence
\begin{equation}
e^{-\sigma|t|+\sigma|\tau|}\leq{}e^{\sigma|t-\tau|},\quad\forall\,t\in\mathbb{R},\,\tau\in\mathbb{R}. \end{equation}
From this inequality and from the expression \eqref{AHF} for \(C_{\omega}\) we conclude that
\begin{equation}
\label{EKc}
|K_c(t,\tau)\leq k^c_{\sigma,\omega}(t-\tau),\quad\forall\,t\in\mathbb{R},\,\tau\in\mathbb{R},
\end{equation}
where
\begin{equation}
\label{EI}
k^c_{\sigma,\omega}(\xi)=\frac{\omega}{\pi}\frac{e^{\sigma|\xi|}}{\cosh\omega\xi}\ccomma
\quad\xi\in\mathbb{R}.
\end{equation}
For \(0\leq\sigma<\omega\), the function \(k^c_{\sigma,\omega}\) belongs to \(L^1\) and
\begin{equation}
\label{Eskc}
\|k_{\sigma,\omega}^c\|_{L_1}<
\frac{4}{\pi}\int\limits_{0}^{\infty}\frac{\cosh{}a\xi}{\cosh\xi}d\xi\,\,\ccomma
\end{equation}
where
\begin{equation}
\label{Esksa}
a=\frac{\sigma}{\omega}\cdot
\end{equation}
The integral in \eqref{Eskc} can be calculated explicitly:
\begin{equation}
\label{Exc}
\int\limits_{0}^{\infty}\frac{\cosh{}a\xi}{\cosh\xi}d\xi=\frac{\pi}{2\cos\frac{\pi}{2}a}\cdot
\end{equation}
Thus
\begin{equation}
\label{fExc}
\|k_{\sigma,\omega}^c\|_{L_1}<
\frac{2}{\sin\frac{\pi}{2}(1-a)}\cdot
\end{equation}
Since \(\sin\frac{\pi}{2}\eta\geq\eta\) for \(0\leq\eta\leq1\), the inequality \eqref{fExc} implies the inequality
\begin{equation}
\label{Nokc}
\|k_{\sigma,\omega}^{c}\|_{L^1}<\frac{2}{1-\sigma/\omega}.
\end{equation}
From \eqref{Rew} and \eqref{EKc} we obtain the inequality \(|v(t)|\leq w(t),\ \forall\,t\in\mathbb{R}\), and
\begin{equation}
\label{InI}
\|v\|_{L^2}\leq \|w\|_{L^2},
\end{equation}
where
\begin{equation}
\label{Inf}
w(t)=\int\limits_{\mathbb{R}}k_{\sigma,\omega}^{c}(t-\tau)|u(\tau)|d\tau.
\end{equation}
According to Lemma \ref{CoT}, \(w\in{}L^2\), and the inequality
\begin{equation}
\label{InE}
\|w\|_{L^2}\leq\|k_{\sigma,\omega}^{c}\|_{L^1}\cdot\|u\|_{L^2}
\end{equation}
holds. The inequality
\begin{equation*}
\|v\|_{L^2}\leq\frac{2}{1-\sigma/\omega}\,\|u\|_{L_2}
\end{equation*}
is a consequence of the equalities \eqref{InI}, \eqref{InE} and \eqref{Nokc}.
According to \eqref{Sc}, \(\|u\|_{L^2}=\|f\|_{L^2_{\sigma}}\), \(\|v\|_{L^2}=\|g\|_{L^2_{\sigma}}\). So the inequality \eqref{BOc} holds with \(M_c=2\).
\end{proof}
\begin{lem}
\label{Bos}
Assume that \(0\leq\sigma<\omega\). Let  \(f\in{}L^2_{\sigma}\), and \(g\) is related to \(f\)
by means of the formula \eqref{fgs}, i.e. \(g=\boldsymbol{S}_{\omega}f\). Then
\(g\in{}L^2_{\sigma}\), and
\begin{equation}
\label{BOs}
\|g\|_{L^2_{\sigma}}\leq\frac{M_s}{(1-\sigma/\omega)^2}\,\|f\|_{L^2_{\sigma}},
\end{equation}
where \(M_s<\infty\) is a value which does not depend on  \(\omega\) and \(\sigma\).
\end{lem}
\begin{proof} Let \(u(\tau)\), \(v(t)\) be defined according to \eqref{Sc}.
Since \(f\in{}L^2_{\sigma}\), \(u\in{}L^2\).
The equality \eqref{fgs} can be rewritten as
\begin{equation}
\label{Rews}
v(t)=\mathrm{p.v.}\int\limits_{\mathbb{R}}e^{-\sigma|t|+\sigma|\tau|}S_{\omega}(t-\tau)u(\tau)d\tau.
\end{equation}
We present \(v(t)\) as
\begin{equation}
\label{Spls}
v(t)=v_1(t)+v_2(t),
\end{equation}
where
\begin{equation}
\label{v1}
v_1(t)=\mathrm{p.v.}\int\limits_{\mathbb{R}}S_{\omega}(t-\tau)u(\tau)d\tau,
\end{equation}
\begin{equation}
\label{v2}
v_2(t)=\int\limits_{\mathbb{R}}(e^{-\sigma|t|+\sigma|\tau|}-1)
S_{\omega}(t-\tau)u(\tau)d\tau,
\end{equation}
Let us estimate the kernel
\begin{equation}
\label{Ks}
K_s(t,\tau)=(e^{-\sigma|t|+\sigma|\tau|}-1)S_{\omega}(t-\tau).
\end{equation}
From the inequalities \(|e^{\xi}-1|\leq|\xi|e^{|\xi|}\), from \eqref{TrI}
 and from the expression \eqref{AHF} for \(S_{\omega}\) we conclude that
\begin{equation}
\label{EKs}
|K_s(t,\tau)|\leq{}k^s_{\sigma,\omega}s(t-\tau),
\end{equation}
where
\begin{equation}
\label{ks}
k^s_{\sigma,\omega}(\xi)=\frac{\omega}{\pi}\cdot\frac{\sigma|\xi|\,e^{\sigma|\xi|}}{\sinh\omega|\xi|}\cdot
\end{equation}
The function \(k^s_{\sigma,\omega}\) belongs to \(L^1\), and
\begin{equation}
\label{Esks}
\|k^s_{\sigma,\omega}\|_{L_1}<\frac{4a}{\pi}\int\limits_{0}^{\infty}\frac{\xi\,\cosh{}a\xi}{\sinh\xi}\ccomma
\end{equation}
where \(a\) is the same that in \eqref{Esksa}. The integral in \eqref{Esks} can be calculated
explicitly:
\begin{equation}
\label{Exs}
\int\limits_{0}^{\infty}\frac{\xi\,\cosh{}a\xi}{\sinh\xi}d\xi=
\frac{{\pi}^2}{4\sin^2\frac{\pi}{2}(1-a)}\cdot
\end{equation}
Thus the inequality
\begin{equation}
\label{Noks}
\|k_{\sigma,\omega}^{s}\|_{L^1}<\frac{\pi}{(1-\sigma/\omega)^2}
\end{equation}
holds. From \eqref{v2}, \eqref{Ks} and \eqref{EKs} it follows that
\begin{equation}
\label{InIs}
\|v_2\|_{L^2}\leq \|w\|_{L^2},
\end{equation}
where
\begin{equation}
\label{Infs}
w(t)=\int\limits_{\mathbb{R}}k_{\sigma,\omega}^{s}(t-\tau)|u(\tau)|d\tau.
\end{equation}
According to Lemma \ref{CoT}, \(w\in{}L^2\), and the inequality
\begin{equation}
\label{InEs}
\|w\|_{L^2}\leq\|k_{\sigma,\omega}^{s}\|_{L^1}\cdot\|u\|_{L^2}
\end{equation}
holds. From \eqref{Noks}, \eqref{InEs} and \eqref{InIs} we conclude that
\begin{equation}
\label{v2es}
\|v_2\|_{L^2}\leq\frac{\pi}{(1-\sigma/\omega)^2}\|u\|_{L^2}.
\end{equation}
According to Lemma \ref{Contr}, the inequality
\begin{equation}
\label{Sip}
\|v_1\|_{L^2}\leq\|u\|_{L^2}
\end{equation}
holds. From \eqref{Spls}, \eqref{Sip} and \eqref{v2es} we derive the inequality
\begin{equation}
\label{DIn}
\|v\|_{L^2}\leq\frac{M_s}{(1-\sigma/\omega)^2}\|u\|_{L^2}
\end{equation}
with \(M_s=\pi+1\).
\end{proof}
\section{The Akhiezer operators \(\boldsymbol{\Phi}_{\omega}\) and \(\boldsymbol{\Psi}_{\omega}\) in \(\boldsymbol{L^2_\sigma\oplus{}L^2_\sigma}\).}
\begin{defn}
\label{OrtSs}
\emph{The space} \(L^{2}_\sigma\oplus{}L^{2}_\sigma\) is the set of all \(2\times1\) columns
\(\boldsymbol{x}=\begin{bmatrix}x_1\\ x_2\end{bmatrix}\) such that
\(x_1(t)\in{}L^{2}_\sigma\) and \(x_2(t)\in{}L^{2}_\sigma\), where \(L^2_\sigma\)
was defined in Definition~\ref{Lsig}. The set
\(L^{2}_\sigma\oplus{}L^{2}_\sigma\) is equipped by the natural linear operations
and by the norm
\begin{equation}
\label{NoPs}
\|\boldsymbol{x}\|_{L^{2}_\sigma\oplus{}L^{2}_\sigma}=
\sqrt{\|x_1\|^2_{L^2_\sigma}+\|x_2\|^2_{L^2_\sigma}}.
\end{equation}
\end{defn}
Since\footnote{See \eqref{subb}.} \(L^2_\sigma\subset{}L^1_{\omega}\), also \(L^2_\sigma\oplus{}L^2_\sigma\subset{}L^1_{\omega}\dotplus{}L^1_{\omega}\). Thus if
\(\boldsymbol{x}\in{}L^2_\sigma\oplus{}L^2_\sigma\), then the values
\begin{math}
\boldsymbol{y}(t)=(\boldsymbol{\Phi}_{\omega}\boldsymbol{x})(t)
\end{math}
and
\begin{math}
\boldsymbol{z}(t)=(\boldsymbol{\Psi}_{\omega}\boldsymbol{x})(t)
\end{math}
are defined by \eqref{Ent} for almost every \(t\in\mathbb{R}\). From Lemmas \ref{Boc} and \ref{Bos} we derive
\begin{lem}
\label{Bl2s}
We assume that \(0\leq\sigma<\omega\).
Let \(\boldsymbol{x}\in{}L^2_\sigma\oplus{}L^2_\sigma\) and let
\begin{math}
\boldsymbol{y}=\boldsymbol{\Phi}_{\omega}\boldsymbol{x},
\end{math}
\begin{math}
\boldsymbol{z}=\boldsymbol{\Psi}_{\omega}\boldsymbol{x}
\end{math}
be defined by \eqref{Ent}.
Then \(\boldsymbol{y}\in{}L^2_\sigma\oplus{}L^2_\sigma\), \(\boldsymbol{z}\in{}L^2_\sigma\oplus{}L^2_\sigma\), and the estimates hold
\begin{align}
\label{Bl2si1}
\|\boldsymbol{\Phi}_{\omega}\boldsymbol{x}\|_{L^2_\sigma\oplus{}L^2_\sigma}&\leq
\tfrac{M}{1-\sigma/\omega}\,\|\boldsymbol{x}\|_{L^2_\sigma\oplus{}L^2_\sigma},\\
\label{Bl2si2}
\|\boldsymbol{\Psi}_{\omega}\boldsymbol{x}\|_{L^2_\sigma\oplus{}L^2_\sigma}&\leq
\tfrac{M}{1-\sigma/\omega}\,\|\boldsymbol{x}\|_{L^2_\sigma\oplus{}L^2_\sigma},
\end{align}
where \(M<\infty\) is a value which does not depend on \(\sigma,\,\omega, \,\boldsymbol{x}\).
\end{lem}
The following theorem is a main result of this paper.
\begin{thm}
We assume that \(0\leq\sigma<\omega\). Then for every \(\boldsymbol{x}\in{}L^2_\sigma\oplus{}L^2_\sigma\) the equalities
\begin{equation}
\label{MuIns}
\boldsymbol{\Psi}_{\omega}\boldsymbol{\Phi}_{\omega}\boldsymbol{x}=\boldsymbol{x},\quad
\boldsymbol{\Phi}_{\omega}\boldsymbol{\Psi}_{\omega}\boldsymbol{x}=\boldsymbol{x}
\end{equation}
hold.
\end{thm}
\begin{proof} From Lemma \ref{Bl2s} it follows that the operators \(\boldsymbol{\Psi}_{\omega}\boldsymbol{\Phi}_{\omega}\) and
\(\boldsymbol{\Phi}_{\omega}\boldsymbol{\Psi}_{\omega}\) are
bounded linear operators in the space \(L^2_\sigma\oplus{}L^2_\sigma\).
The set \(L^2\oplus{}L^2\) is a dense subset of the space \(L^2_\sigma\oplus{}L^2_\sigma\).
By Theorem \ref{muin}, the equalities \eqref{MuIns} holds for every
\(\boldsymbol{x}\in{}L^2\oplus{}L^2\). By continuity, the equalities \eqref{MuIns} can be
extended from  \(L^2\oplus{}L^2\) to \(L^2_\sigma\oplus{}L^2_\sigma\).
\end{proof}

% ------------------------------------------------------------------------

\begin{thebibliography}{1111}

%%%%%%%%%%%%%%%%%%%%%%%%%%%%%%%%%%%%%%%%%%%%%%%%%%%%%%%%%%%%%%%%%%%%%%
\bibitem{A} N.I.\,Akhiezer. \textit{Lectures on Integral Transforms.} Amer. Math. Soc.,
Providence, RI. 1985.
%%%%%%%%%%%%%%%%%%%%%%%%%%%%%%%%%%%%%%%%%%%%%%%%%%%%%%%%%%%%%%%%%
%%%%%%%%%%%%%%%%%%%%%%%%%%%%%%%%%%%%%%%%%%%%%%%%%%%%%%%%%%%%%%%%%%%%%%%%%%%%
\bibitem{T} E.C.\,Titchmarch. \textit{Introduction to the Theory
of Fourier Integrals.} Clarendon Press, Oxford 1937. x+394pp.
Third Edition. Chelsea Publishing Co., New\,York 1986.
%%%%%%%%%%%%%%%%%%%%%%%%%%%%%%%%%%%%%%%%%%%%%%%%%%%%%%%%%%%%%%%%
\bibitem{Bo} V.I.\, Bogachev. \textit{Measure Theory. Vol.1}. Springer-Verlag,
Berlin-Heidelberg-New York, 2007.
\end{thebibliography}
\end{document}